\documentclass[final,12pt]{elsarticle}
\usepackage{xcolor,soul,cancel}
\usepackage[color]{showkeys}
\definecolor{refkey}{rgb}{0,1,1}
\definecolor{labelkey}{rgb}{1,0,0}
\usepackage{amssymb}
\usepackage{amsmath}
\usepackage{amsthm}

\journal{arXiv}

\newtheorem{rem}{Remark}
\newtheorem{thm}{Theorem}
\newtheorem{lem}[thm]{Lemma}
\newtheorem{cor}{Corollary}

\numberwithin{equation}{section}
\newcommand{\eq} [1] {\begin{equation}\label{#1}\quad}
\newcommand{\en} {\end{equation}}

\newcommand{\norm}[1]{\left\Vert#1\right\Vert}
\newcommand{\abs}[1]{\left\vert#1\right\vert}

\newcommand{\ind}{\operatorname{Ind}}

\newcommand{\Z}{\mathbb Z}
\newcommand{\C}{\mathbb C}

\newcommand{\R}{\mathbb R}

\newcommand{\diag}{\operatorname{diag}}
\newcommand{\adj}{\operatorname{adj}}

\newcommand{\im}{\operatorname{Im}}
\newcommand{\coker}{\operatorname{coker}}
\newcommand{\Ind}{\operatorname{Ind}}
\begin{document}

\begin{frontmatter}

\title{One sided invertibility of matrices over commutative rings, corona problems, and Toeplitz operators with matrix symbols\tnoteref{support}}
\author[C]{M. C. C\^{a}mara}
\ead{cristina.camara@math.ist.utl.pt}
\address[C]{Departamento de Matem\'atica, CAMGSD, Instituto Superior T\'ecnico,
Technical University of Lisbon (TULisbon), Portugal}

\tnotetext[support]{The work on this paper was started during the visits of Rodman and Spitkovsky to
Instituto Superior T\'ecnico, supported by the FCT project PTDC/MAT/81385/2006 (Portugal). The first author was partially supported by FCT through the Program POCI 2010/FEDER and the project PTDC/MAT/121837/2010}

\author[W]{L. Rodman}
\ead{lxrodm@math.wm.edu}
\address[W]{Department of Mathematics, College of William and Mary, Williamsburg, VA 23187, USA}

\author[W]{I. M. Spitkovsky\corref{cor}}
\ead{ilya@math.wm.edu, imspitkovsky@gmail.com}

\cortext[cor]{Corresponding author.}

\begin{abstract}
Conditions are established under which Fredholmness, Coburn's property and one- or two-sided invertibility are shared by a Toeplitz operator with matrix symbol $G$ and the Toeplitz operator with scalar symbol $\det G$. These results are based on one-sided invertibility criteria for  rectangular matrices over appropriate commutative rings and related scalar corona type problems.
\end{abstract}


\end{frontmatter}

\section{Introduction}

To outline the main topics of this paper, we need first to agree on some standard notation and introduce some terminology. For any set $X$, we will denote by $X^{n \times m}$ the set of $n \times m$ matrices with entries in $X$, abbreviating $X^{n \times 1}$ to $X^n$. If $X$ is a Banach space (a ring, a (Banach) algebra), then $X^{n\times m}$ (resp., $X^{n\times n}$) is also supplied with the Banach space (resp., ring, (Banach) algebra) structure. A diagonal matrix from $X^{n\times n}$ with the diagonal entries $x_1,\ldots,x_n$ will be denoted $\diag[x_1,\ldots,x_n]$.

Some important examples of $X$ include the Lebesgue spaces $L_p(\R)$ of functions defined on the real line $\R$, and their subspaces $H_p^\pm$ of the traces on $\R$ of the functions from the Hardy spaces\footnote{$H_\infty^\pm$ consist of all functions analytic and bounded in $\C^\pm$; see Section~\ref{s:Fred} for the precise definition of $H_p^\pm$ for $p<\infty$.} over the half-planes $\C^\pm:=\{z\in\C\colon \pm\im z>0\}$. We also denote by $M^\pm_p$, $1\leq p\leq\infty$, the linear sums of $H^\pm_p$ with the algebra $\mathcal R$ of all rational functions in $L_\infty(\R)$: \[ M^\pm_p=H^\pm_p+{\mathcal R}. \] The closure of $\mathcal R$ in the uniform norm is the algebra $C$ of all functions continuous on the one point compactification $\dot\R:=\R\cup\{\infty\}$ of $\R$, while the closure of $M^\pm_\infty$ coincides with $H^\pm_\infty +C$. The latter is thus a (closed) subalgebra of $L_\infty(\R)$. Finally, for any ring $\mathcal A$, we let $\mathcal GA$ stand for the set of its invertible elements.

Recall that a bounded linear operator $A\colon X\rightarrow Y$ acting between Banach spaces $X$ and $Y$ is {\em Fredholm} if its kernel $\ker A$ and cokernel $\coker A= Y/\im A$ are finite dimensional. Note that then $\dim\coker A=\dim\ker A^*$; the difference \[ \ind A=\dim\ker A-\dim\coker A \] is the {\em (Fredholm) index} of $A$. We say that two operators $A\colon X \rightarrow Y$ and $\widetilde{A}\colon \widetilde{X} \rightarrow \widetilde{Y}$ are {\em Fredholm equivalent} if either they are both Fredholm, with the same Fredholm index, or they are both non-Fredholm. Further, $A$ and $\widetilde{A}$ are {\em nearly Fredholm equivalent} if they are both Fredholm (with no relation imposed on their indices) or both non-Fredholm, and {\em strictly Fredholm equivalent} if they are Fredholm equivalent and, in case they are both Fredholm operators, $\dim\ker A=\dim\ker\widetilde{A}$ and $\dim\coker A=\dim\coker\widetilde{A}$ their kernels have the same dimensions, and their cokernels have the same dimension.

We are ultimately interested in Fredholm properties of {\em Toeplitz operators} $T_G$ with matrix symbols $G\in L_\infty^{n\times n}$ acting on $(H_p^+)^n$, $1<p<\infty$, and in particular their  relations with those of $T_{\det G}$.

Observe first of all that for $G\in (H^\pm_\infty+C)^{n\times n}$, the operators $T_G$ and $T_{\det G}$ are Fredholm equivalent. This follows directly from the Fredholmness criterion and index formula from \cite{Dou68}, see also \cite{Dou98} or \cite[Section 5.1]{LS}. Indeed, $T_G$ is Fredholm if and only if $G\in {\mathcal G}(H^\pm_\infty+C)^{n\times n}$, which in turn happens if and only if $\det G\in {\mathcal G}(H^\pm_\infty+C)$. Under this condition, $\ind T_G$ coincides, up to the sign, with the winding number of the harmonic extension of $\det G$ into $\C^\pm$ along horizontal lines sufficiently close to $\R$. The particular case of $G\in C^{n\times n}$ is of course simpler \cite{MP86}: $T_G$ is then Fredholm if and only if $G\in {\mathcal G}C$, while $\ind T_G=\ind T_{\det G}$ is the opposite of the winding number of $\det G$ over $\R$. On the other hand, already for piecewise continuous $G$ with just one point of discontinuity, starting with $n=2$, there are examples of both nearly Fredholm (but not Fredholm) and not even nearly Fredholm equivalent operators $T_G$, $T_{\det G}$. These examples can be easily constructed, based on the Fredholm crtierion and index formula for Toeplitz operators with (matrix) piecewise continuous symbols, see e.g. \cite{BSil06,CG,LS}.

Now suppose that \eq{i1.1} G=M_-G_0M_+^{-1}, \en where $M_\pm\in {\mathcal G}(H^\pm_\infty+C)^{n\times n}$. Then the Toeplitz operators $T_G$ and $T_{G_0}$ are nearly Fredholm equivalent \cite[Theorem 5.5]{LS}. These two operators are strictly Fredholm equivalent if $M_\pm\in {\mathcal G}(H^\pm_\infty)^{n\times n}$ \cite{LS,MP86}.

In particular, if $G_0=I$ in \eqref{i1.1}, we conclude from here that the operator $T_G$ is Fredholm whenever \eq{i1.3} GM_+=M_-\en for some $M_\pm\in {\mathcal G}(M^\pm_\infty)^{n\times n}$. We remark that, in this case, $\det G$ admits an analogous (scalar) representation \eq{i1.4} \det G= (\det M_-)(\det M_+)^{-1}, \en and $T_{\det G}$ is also Fredholm.

Note also that each of the $n$ columns of $M_+$, together with the corresponding column of $M_-$, due to \eqref{i1.3} yield a solution to the {\em Riemann-Hilbert problem} \eq{i1.5} G\Phi_+=\Phi_-, \quad \Phi_\pm\in (M^\pm_\infty)^n.\en

In the $2\times 2$ case it was shown in \cite{CDR} that only one solution to \eqref{i1.5} is needed to conclude that $T_G$ is Fredholm equivalent to $T_{\det G}$ --- as it happens e.g. in the case of continuous symbols, --- as long as $\Phi_\pm$ are {\em corona-type pairs}, left invertible over $M^\pm_\infty$. Moreover, the Fredholm equivalence is strict if $\Phi_\pm$ are left invertible over $H^\pm_\infty$, i.e., satisfy the corona condition in the corresponding half-planes $\C^\pm$. Thus it is possible to reduce the study of the Fredholm properties of Toeplitz operators with a $2\times 2$ matrix symbol to the study of analogous properties for a Toeplitz operator with a scalar symbol. The following question is then natural to ask: is it possible to generalize the results of \cite{CDR} to $n\times n$ matrix symbols if, instead of $n$ solutions to \eqref{i1.5}, we have $n-1$ solutions satisfying some form of a left invertibility condition? Or if we have an $n\times n$ symbol whose elements are continuous on $\dot\R$ except for one column or a row?

Left invertibility of $n\times m$ ($m\leq n$) matrix functions over $H^+_\infty$ was studied in \cite[Theorem 3.1]{Fuhr} where a generalization of the Carleson corona theorem to the case of matrix valued analytic functions was obtained by reduction to the scalar corona theorem via determinants. There it was shown, in particular, that if the determinants of all $m\times m$ submatrices satisfy a (scalar) corona condition, then the $n\times m$ matrix function is left invertible over $H^+_\infty$. We may therefore ask: is this also a necessary condition? Can we deduce analogous results in the more abstract context of a unital commutative ring, allowing a broader range of applications, and obtain expressions for the left inverses in terms of the solutions to an associated scalar corona-type problem?

We now turn to a different question concerning Toeplitz operators.  In the scalar case, they possess what is known as  {\em Coburn's property}, first observed in the Hilbert space (that is, for $p=2$) setting in the proof of Theorem~4.1 in \cite{Co}: for any Toeplitz operator with non-zero symbol $g\in L_\infty(\R)$, $T_g$ or $T^*_g$ has a zero kernel. It follows, in particular, that Fredholmness of $T_g$ implies its one-sided invertibility.

In this respect the situation is quite different when the symbol is matricial. The latter case presents much greater difficulties, some of which are naturally due to the non-commutativity of multiplication and the impossibility of division by vectorial functions. The degree of difficulty increases with the order of the matrix symbols involved, as reflected by the overwhelmingly greater number of results and papers concerning Toeplitz operators and related problems for $2\times 2$ symbols, as compared with the general $n\times n$ case.

Coburn's property is among many familiar properties holding in the scalar setting but not, in general, in the matricial setting, even in the simplest case when the symbol is diagonal. A natural question thus arises: what classes of Toeplitz operators with matricial symbols satisfy Coburn's property?

In this paper we address and relate all these apparently different questions, taking an algebraic point of view that enables us to unify and tackle different problems in different settings. This approach provides moreover a good illustration of how the study of Toeplitz operators knits together different areas of mathematics such as operator theory, complex analysis, and algebra.

The paper is organized in the following way. Section~\ref{s:osi} contains one-sided invertibility criteria for rectangular matrices with elements from an abstract commutative ring $\mathcal A$, along with formulas for the respective inverses. In Section~\ref{s:corona}, these results are recast for ${\mathcal A}$ being $H^\pm_\infty$ or $M^\pm_\infty$ with the help of corresponding corona theorems. These results are used in the main Section~\ref{s:osift}, where conditions are established on matrix functions $G$ guaranteeing that $T_G$ and $T_{\det G}$ are (nearly or strongly) Fredholm equivalent. It is  preceded by a short Section~\ref{s:Fred}  containing the necessary background information on the relations between Fredholmness of $T_G$ and factorization of $G$. Some special cases (unitary or orthogonal matrices, and matrix functions continuous except for one row or column) are considered in Section~\ref{s:Final}. Finally, in Section~\ref{s:ap} we deal with almost periodic symbols $G$.

\section{One sided invertibility of matrices over commutative rings}\label{s:osi}

In this section, ${\mathcal  A}$ is a unital commutative ring. We say that an element $a\in {\mathcal A}^{n\times k}$, $k\leq n$, is {\em left invertible over $\mathcal A$} if there exists $b\in {\mathcal A}^{k\times n}$ such that $ba=I_k$, the identity matrix in ${\mathcal A}^{k\times k}$. The notion of right invertibility over $\mathcal A$ is introduced in a similar way. The treatment of (one sided) invertiblity of square matrices with elements in $\mathcal A$ can be found in \cite[Chapter I]{Kru87}. We are interested in the case of ${\mathcal A}^{n\times k}$ with $k\neq n$.

For any matrix $\Phi\in {\mathcal A}^{n\times n}$ and $I\subset \{1,\ldots, n\}$, $\Phi_I$ will stand for its submatrix  obtained by keeping the $i$-th rows, $i\in I$ while deleting all other rows. If $m\leq n$, label by $I_1, I_2, \ldots, I_N$ all  $N={n\choose m}$ subsets of $\{1,\ldots,n\}$ with $m$ elements, and denote  $d_k^\Phi:=\det \Phi_{I_k}$.

\begin{lem}\label{l:C4.1} Let $\Phi\in {\mathcal A}^{n\times m}$ with $m\leq n$, and let $\Phi_I$ be some $m\times m$ submatrix of $\Phi$. Denote by $\Delta_{pq}^{\Phi_I}$ the determinant of the matrix obtained from $\Phi_I$ by deleting its $p$-th row and $q$-th column. Define $\Phi_I^*\in {\mathcal A}^{m\times n}$ by setting its $(q,p)$-entry according to the formula \eq{C2} \Phi_{qp}^*=\begin{cases} (-1)^{p+q}\Delta_{pq}^{\Phi_I} & \text{ if } p\in I, \\
0 & \text{ otherwise.}\end{cases} \en Then \[ \Phi_I^*\Phi = \det\Phi_I \diag[ (-1)^q]_{q=1,\ldots m}. \] \end{lem}
\begin{proof}By construction, \[ \sum_{p=1}^m \Phi_{qp}^*\Phi_{pq}=(-1)^q\det\Phi_I, \quad q=1,\ldots,m,\] while $\sum_{p=1}^m \Phi_{qp}^*\Phi_{pl}=0$ for $q\neq l$ as the determinant of a matrix with two coinciding columns.  \end{proof}

\begin{thm}\label{th:C4.2} {\em (i)} An element $\Phi$ of ${\mathcal A}^{n\times m}$ is left invertible over $\mathcal A$ if and only if the column \[ \Delta:=
\begin{bmatrix} d_1^\Phi\\ d_2^\Phi \\ \vdots \\ d_N^\Phi\end{bmatrix} \] is left invertible in $\mathcal A$.

{\em (ii)} If $\Phi\in {\mathcal A}^{n\times m}$ is left invertible over $\mathcal A$ with left inverse $\Psi\in{\mathcal A}^{m\times n}$, then the row \[ \Delta^*=[d_1^{\Psi^T}, d_2^{\Psi^T}, \ldots, d_N^{\Psi^T}] \] is a left inverse of $\Delta$  over $\mathcal A$.

{\em (iii)} If $\Delta$ is left invertible in $\mathcal A$ with a left inverse $\Delta^*=[\Delta_1^*,\Delta_2^*, \ldots, \Delta_N^*]$, then \[ \Psi =\diag [(-1)^q]_{q=1,\ldots,n}\sum_{k=1}^N\Delta_k^*\Phi_{I_k}^*,  \] where $\Phi_{I_k}^*$ are defined in accordance with \eqref{C2} with $I=I_k$, is a left inverse of $\Phi$.    \end{thm}
\begin{proof} (ii) If $\Psi\Phi=I_m$, then the Cauchy-Binet formulas show that $\Delta^*\Delta=1$.

(iii) We have \[ \Psi\Phi =\diag [(-1)^q]_{q=1,\ldots,n}\sum_{k=1}^N\Delta_k^*\Phi_{I_k}^*\Phi  \] which, by Lemma~\ref{l:C4.1}, is equal to
\[ \diag [(-1)^q]_{q=1,\ldots,n}\sum_{k=1}^N(\Delta_k^* d_{I_k}^\Phi) \diag [(-1)^q]_{q=1,\ldots,n}=I_m.\]

(i) is an immediate consequence of (ii) and (iii). \end{proof}

The ``if'' part of Theorem~\ref{th:C4.2} is an abstract version of its particular case when $\mathcal A$ is the algebra of bounded analytic functions on the unit disc contained in the proof of \cite[Theorem 3.1]{Fuhr}; the ``only if'' part shows moreover that the converse is true.

In what follows, we adapt the notation to the special case $m=n-1$ which is of particular relevance to  the main results of the paper.

Given $\Phi \in {\mathcal A}^{n \times (n-1)}$ we denote by $\Delta_{p; \cdot} (\Phi)$ the
determinant of the $(n-1) \times (n-1)$ matrix obtained by omitting the
row $p$ in $\Phi$; we denote by $\Delta_{p,s;j} (\Phi)$ the
determinant of the $(n-2) \times (n-2)$ submatrix of $\Phi$ obtained
by omitting the rows $p$ and $s$ ($p\neq s$) and column $j$ (we
take $p,s\in \{1,2,\ldots, n\}$, $j\in \{1,2,\ldots, n-1\}$). Analogously, for
$\Psi \in   {\mathcal A}^{(n-1) \times n}$, we use the notation
$\Delta_{\cdot; p}(\Psi)$ for the determinant of the  $(n-1) \times
(n-1)$ matrix obtained by omitting the column  $p$ in $\Psi$; and
$\Delta_{j;p,s} (\Phi)$ stands for the determinant of the $(n-2) \times
(n-2)$ submatrix of $\Psi$ obtained by omitting the columns $p$ and
$s$ ($p\neq s$) and row $j$.

\begin{cor}\label{aug86}
An element $\Phi\in {\mathcal  A}^{n \times (n-1)}$ is left invertible
over ${\mathcal  A}$ if and only if  the column
\begin{equation}\label{aug83} \left[\begin{array}{c}\Delta_{1;\cdot}(\Phi) \\  \vdots\\
\Delta_{n;\cdot}(\Phi)\end{array}\right] \end{equation}  is left invertible over
${\mathcal  A}$.

Moreover, in this case a left inverse of $\Phi$ is given by
\begin{equation}\label{I.3} \Psi=\left[\begin{array}{c} \Psi_1 \\ \Psi_2 \\ \vdots \\
\Psi_{n-1}\end{array}\right],
\qquad \Psi_j \in {\mathcal  A}^{1 \times n}, \quad j=1,2,\ldots ,n-1, \end{equation}
with
\eq{I.4}
\Psi_j=(-1)^j \Delta^*
\left[\begin{array}{ccccc} 0 & \Delta_{1,2;j}  &
\Delta_{1,3;j}& \ldots & \Delta_{1,n;j} \\
-\Delta_{1,2;j} & 0& \Delta_{2,3;j}& \ldots & \Delta_{2,n;j} \\
-\Delta_{1,3;j}& -\Delta_{2,3;j} & 0& \ldots & \Delta_{3,n;j} \\
\vdots & \vdots & \vdots & \ddots & \vdots  \\
-\Delta_{1,n;j} & -\Delta_{2,n;j}& -\Delta_{3,n;j}& \ldots & 0 \end{array}\right] \, \cdot \,
\widetilde{I}_n, \en
for $j=1,2,\ldots, n-1$, where
$\Delta_{p,s;j}:= \Delta_{p,s;j}(\Phi)$,
\eq{I.5}
\Delta^*=\left[ \Delta_{1;.}(\Psi^T), \ldots \Delta_{n;.}(\Psi^T)\right] \en
is a left inverse of {\em \eqref{aug83}} over ${\mathcal  A}$, and
\begin{equation}\label{I.6}
 \widetilde{I}_n ={\rm diag}\, [1, -1, 1, \ldots, (-1)^{n+1}]. \end{equation}
\end{cor}

The following result will be crucial in establishing relations relations between left invertibility of some
matrix functions and Fredholmness of Toeplitz operators.

\begin{thm}\label{aug111}
Let $\Phi\in {\mathcal  A}^{n \times (n-1)}$ be left invertible over
${\mathcal A}$, and let $\Psi\in {\mathcal  A}^{(n-1) \times n}$ be its
left inverse:
\begin{equation}\label{I.15} \Psi\Phi=I_{n-1}. \end{equation}
Let moreover
\begin{equation}\label{I.16}
\Phi_e=\left[\Phi \ \ N \right], \quad
\Psi_e=\left[\begin{array}{c} \Psi \\ \widetilde{N} \end{array}\right], \quad
\mbox{with} \quad N\in {\mathcal  A}^{n \times 1}, \ \
\widetilde{N}\in {\mathcal
A}^{1 \times n}.
\end{equation}
Then:
\begin{itemize}
\item[(i)] if
\begin{equation}\label{I.18}
N=\left[\begin{array}{c} N_1\\ N_2\\ \vdots \\ N_n\end{array}\right],\quad N_j=(-1)^{j-1} \Delta_{\cdot;j}(\Psi),
\end{equation}
then
\begin{equation}\label{I.19}
\Psi_e\Phi_e=\left[\begin{array}{cc} I_{n-1} & 0_{(n-1)\times 1} \\
\widetilde{N}\Phi & \widetilde{N}N
\end{array}\right].
\end{equation}
\item[(ii)]
if
\begin{equation}\label{I.21}
\widetilde{N}=\left[\widetilde{N}_1 \ \ \widetilde{N}_2 \  \ \ldots \ \
\widetilde{N}_n \right] \quad \mbox{with } \quad \widetilde{N}_j=(-1)^{j-1}\Delta_{j;\cdot}(\Phi),
\end{equation}
then
\begin{equation}\label{I.22}
\Psi_e\Phi_e =\left[\begin{array}{cc} I_{n-1} & \Psi N \\ 0_{1\times (n-1)} & \widetilde{N}N
\end{array}\right] .\end{equation} 
\item[(iii)] if $N$ and $\widetilde{N}$ satisfy (\ref{I.18}) and (\ref{I.21}), respectively, then
\begin{equation}\label{I.24}
\Psi_e\Phi_e=\Phi_e\Psi_e=I_n \quad \mbox{and} \quad {\rm det}\, \Phi_e={\rm det}\, \Psi_e =(-1)^{n-1}.
\end{equation}
\end{itemize}
\end{thm}

\begin{proof} (i) Since $\Psi$ is right invertible, $\Psi^T$ is left invertible. Therefore, by Corollary~\ref{aug86}, \[ \begin{bmatrix}\Delta_{1;.}(\Psi^T)\\ \vdots\\ \Delta_{n;.}(\Psi^T)\end{bmatrix} \] is left invertible over $\mathcal A$. Let $c_1,\ldots c_n\in \mathcal A$ be such that $\sum_{j=1}^nc_j\Delta_{j;.}(\Psi^T)=1$. Then, setting \[ \Psi_1= [c_1, -c_2, \ldots, (-1)^nc_n] \] and using cofactor expansion across the first row, we see that $\det \Psi_0=1$, where
$$ \Psi_0:=\left[\begin{array}{c} \Psi_1 \\ \Psi \end{array}\right]. $$
Thus, $\Psi_0$ is invertible, $N$ is the first column of $\Psi_0^{-1}$, and the equality $\Psi N=0$ follows.
Part (ii) is proved analogously. For Part (iii) note that the Cauchy-Binet
formula yields
$$ \sum_{j=1}^n \Delta_{\cdot;j}(\Psi)\Delta_{j;\cdot}(\Phi)=1. $$
Thus, taking into account parts (i) and (ii), we have
$\Phi_e\Psi_e=I_n$. Then also $\Psi_e\Phi_e=I_n$ (this is a general
property of matrices with elements in unital
commutative rings, see e.g. \cite{Artin}). Finally, expanding ${\rm det}\,
\Phi_e$ along the last column, we obtain
\begin{equation}\label{dec121}  (-1)^{n-1} {\rm det}\, \Phi_e = \sum_{j=1}^n \Delta_{\cdot, j}(\Psi) \Delta_{j,\cdot}(\Phi),
\end{equation}
which is equal to the $(n,n)$ entry of the product $({\rm adj}\, \Phi_e)
\, \cdot \, ({\rm adj}\, \Psi_e)$, where we denote by ${\rm adj}\, X\in
{\mathcal A}^{n \times n}$ the algebraic adjoint (adjugate) of a matrix
$X\in  {\mathcal A}^{n \times n}$. Since $\Psi_e$ and $\Phi_e$ are
inverses of each other,  then so are ${\rm adj}\, \Phi_e$
and ${\rm adj}\, \Psi_e$, and
(\ref{dec121}) is equal to $1$, as claimed. \end{proof}

\section{ Corona tuples and one sided invertibility in \boldmath{$H^\pm_\infty$} and \boldmath{$M^\pm_\infty$}}\label{s:corona}

Having the results of Section~\ref{s:osi} in mind, we now establish necessary and sufficient conditions for the left invertibility of $n$-tuples in some concrete unital algebras of interest: $H^\pm_\infty$ and $M^\pm_\infty$.

The {\em corona tuples}, with respect to these algebras, are defined as follows:

$$ HCT^\pm_n:=\left\{[h_1^\pm, h_2^\pm,\ldots, h_n^\pm]\colon h_j^\pm \in H_\infty^\pm \quad \mbox{and}
\quad \inf_{z\in \C^\pm} \left(\sum_{j=1}^n |h_j^\pm (z)|\right)>0\right\}, $$
\[ MCT^\pm_n := \left\{[r_1h_1^\pm, r_2h_2^\pm,\ldots, r_nh_n^\pm]\colon [h_1^\pm, \ldots, h_n^\pm]\in
HCT^\pm_n \text{ and } r_1, \ldots, r_n\in \mathcal  {GR}\right\}. \]

\begin{thm}\label{aug82}
${\rm (a)}$  Let $h_1^\pm, h_2^\pm,\ldots, h_n^\pm\in H_\infty^\pm$. Then  $[h_1^\pm, h_2^\pm,\ldots, h_n^\pm]\in
HCT^\pm_n$ if and only if
$\left[\begin{array}{c} h_1^\pm \\ \vdots \\ h_n^\pm \end{array}\right]$ is left invertible over
$H_\infty^\pm$, i.e. there exist $g_j\in H_\infty^\pm$, $j=1,2,\ldots, n$, such that
$\sum_{j=1}^n g_jh_j=1. $

${\rm (b)}$  The following statements are equivalent for
$h_1^\pm, h_2^\pm,\ldots, h_n^\pm\in M^\pm_\infty$:
\begin{itemize}
\item[(1)]
$[h_1^\pm, h_2^\pm,\ldots, h_n^\pm]\in MCT^\pm_n$;
\item[(2)] There exist $r\in \mathcal  {GR}$ and $[g_1,\ldots, g_n]\in HCT_n^\pm$ such that $h_j^\pm=rg_j$, $j=1,2,\ldots, n$;
\item[(3)]
$\left[\begin{array}{c} h_1^\pm \\ \vdots \\ h_n^\pm \end{array}\right]$ is left invertible over
$M_\infty^\pm$.
\end{itemize}
\end{thm}

Part (a) is the classical corona theorem, going back to Carleson \cite{Carl}. 
When proving (b), the case $p=\infty$ of the following  simple observation is needed:
\eq{aug81}  M^\pm_p= \{s\phi \colon s\in {\mathcal
{GR}}, \ \phi\in H^\pm_p \}. \en For a proof see \cite[Proposition~2.3]{CDR}.

{\sl Proof of Part} (b). We follow here the logic of \cite[Theorem 2.6]{CDR}, where the case $n=2$ was considered.

(1) implies (2):
Let (1) hold, that is, $h_j=s_j\phi_j$, where
$s_j\in \mathcal {GR}$ and $\{ \phi_1,\ldots,\phi_n\}$ is a corona
$n$-tuple in $H^+_\infty$ (the case of $H_\infty^-$ can be treated
along the same lines). Denoting by $\{z_1,\ldots,z_N\}$  the set of all
zeros and poles of $s_1,\ldots,s_n$ in $\C^+$ and by $U_\epsilon$ its
$\epsilon$-neighborhood, observe that $h_1,\ldots,h_n$ are analytic,
bounded, and satisfy the corona condition on $\C^+\setminus
U_\epsilon$ for every $\epsilon>0$.  On the other hand, each of the
functions $h_i$ has either a zero or a pole at $z_j$,
$i=1,\ldots,n;j=1,\ldots,N$. Let $\ell_j$ be the minimum of the orders of
all $h_i$ at the given $z_j$ (recall that the order of $\phi$ at $z_0$ is
$k$ (respectively, $-k$) if $z_0$ is a zero (respectively, pole) of $\phi$
with multiplicity $k$.) Introduce \[ s(z)=(z+i)^{-\sum_{j=1}^N\ell_j}\prod_{j=1}^N
(z-z_j)^{\ell_j}.\] Then the functions
$g_i=s^{-1}h_i$ are analytic, bounded and satisfy the corona condition
on $\C^+\setminus U_\epsilon$ simultaneously with $h_i$, because
$s$ is analytic, bounded and bounded away from zero on this set. Due
to the choice of $\ell_j$, we also have that all functions $g_i$ are
analytic on $U_\epsilon$ and for each $j$ at least one of them
assumes a non-zero value at $z_j$. Consequently, $[g_1,\ldots,g_n\in
HCT_n^+$.

(2) implies (3): Let $h_j^\pm=\chi g_j$, $j=1,2,\ldots, n$, for some
$\chi\in \mathcal  {GR}$ and $[g_1,\ldots, g_n]\in HCT_n^\pm$. Using part (a), we
have $\sum_{j=1}^n \ell_jg_j = 1$ for some $\ell_j\in
H^\pm_\infty$. Now $\sum_{j=1}^n (\chi^{-1}\ell_j)h_j^\pm = 1$,
where $\chi^{-1}\ell_j\in M^\pm_{\infty}$ by (\ref{aug81}), and (3)
holds.

(3) implies (1): We have $\sum_{j=1}^n g_jh_j^\pm = 1$ for some
$g_j\in M_\infty^\pm$. By (\ref{aug81}), $h_j^+=r_j\phi_j$ and
$g_j=s_j\psi_j$ for some $r_j, s_j\in \mathcal {GR}$ and
$\phi_j,\psi_j\in H_\infty^\pm$. Consequently, \eq{cc}
\sum_{j=1}^nr_js_j\phi_j\psi_j = 1 \text{ on } \C^\pm.\en Without
loss of generality we may suppose that the functions $\phi_j$ do not
all vanish simultaneously  at any point in some open set
$\Omega\subseteq \C^\pm$ containing all the poles of $r_js_j$,
$j=1,\ldots,n$, in the upper half plane,  since otherwise a respective
rational factor could be moved from $\phi_j$ to $r_j$. Then the
$n$-tuple $[\phi_1,\ldots,\phi_n]$ satisfies the corona condition on
$\Omega$.

Since $r_js_j$ ($j=1,2,\ldots, n$)are bounded on $\C^\pm\setminus\Omega$, the
corona condition for  $(\phi_1,\ldots,\phi_n)$ follows from \eqref{cc}.
Thus, $[\phi_1,\ldots,\phi_n]\in HCT_n^\pm$.  \qed
\bigskip

Note that (2) implies (1) in a trivial way.

Corollary~\ref{aug86} admits therefore the following interpretation.

\begin{thm}\label{aug93}
${\rm (a)}$ Let $\Phi\in (H_{\infty}^\pm)^{n\times (n-1)}$. Then
$\Phi$ is left invertible over $H_{\infty}^\pm$ if and only if
$\left[\Delta_{1,.}(\Phi),\ldots, \Delta_{n,.}(\Phi)\right]\in HCT^\pm_n$.

${\rm (b)}$  Let $\Phi\in (M_{\infty}^\pm)^{n\times (n-1)}$. Then
$\Phi$ is left invertible over $M_{\infty}^\pm$ if and only if
$\left[\Delta_{1,.}(\Phi),\ldots, \Delta_{n,.}(\Phi)\right]\in MCT^\pm_n. $

In both cases formula {\em\eqref{I.3}} applies, provided $\Phi$ is left invertible over the
respective algebra.
\end{thm}

According to \cite[Theorem 2.7]{CDR}, $\Phi\in (M_{\infty}^+)^{2\times 1}$ is left invertible in $M_\infty^+$ if and only if there exists a matrix function $R\in \mathcal{G}\mathcal{R}^{2\times 2}$ and $f_+\in HCT^+_1$ such that $\Phi=Rf_+$. We here extend this result to include $\Phi\in (M_{\infty}^\pm)^{n\times (n-1)}$ with arbitrary $n\in\mathbb N$.

\begin{thm}\label{aug910}
Let $\Phi\in  (M_{\infty}^\pm)^{n\times (n-1)}$. Then $\Phi$ is left
invertible over $M_{\infty}^\pm$ if and only if there exist $R\in \mathcal{GR}^{n\times n}$,
$Q\in \mathcal{GR}^{(n-1)\times(n-1)}$ and $F\in (H^\pm_\infty)^{n\times(n-1)}$, the latter being left invertible over
$H^\pm_\infty$, such that \eq{4.19} \Phi=RFQ. \en
\end{thm}

\begin{proof} The sufficiency is obvious. When proving necessity, let us consider the case of invertibility over $M^+_\infty$; the case of $M^-_\infty$ can of course be treated in a similar way.

So, let $\Phi$ be left invertible over $M^+_\infty$. Denote by $\Phi_j$ the $j$-th column of $\Phi$: \[ \Phi=[\Phi_1 \Phi_2\ldots \Phi_{n-1}],
\quad \Phi_j\in (M^+_\infty)^{n\times 1}. \] Then $\Phi_j\in MCT^+_n$ and, by Proposition~\ref{aug82}(c), there exist $\widetilde{r}_j\in \mathcal{GR}$ and
$\widetilde{g}^+_j\in HCT^+_n$ such that $\Phi_j=\widetilde{r}_j\widetilde{g}^+_j$. Thus, \eq{4.20} \Phi=\widetilde{G}^+Q, \en
\[ Q =\diag [ \widetilde{r}_1,\ldots,  \widetilde{r}_{n-1}],\]  \eq{4.22} \widetilde{G}^+=[\widetilde{g}^+_1\ \widetilde{g}^+_2\ \ldots\ \widetilde{g}^+_{n-1}]\in (H^+_\infty)^{n\times (n-1)} \en and \[ \Delta_k(\widetilde{G}^+)=\left(\prod_{j=1}^{n-1}r_j^{-1}\right)\Delta_{\cdot;k}(\Phi), \quad k=1,2,\ldots,n.\] On the other hand, by Theorem~\ref{aug86} (or \ref{aug93}(c))
\[ \Delta(\Phi):=[\Delta_{\cdot;1}(\Phi), \Delta_{\cdot;2}(\Phi),\ldots,
\Delta_{\cdot;n}(\Phi)]\in MCT^+_n \] so that, by Proposition~\ref{aug82}(c), there exist $r\in \mathcal{GR}$ and $g^+\in HCT^+_n$ such that \[ \Delta(\Phi)=rg^+.\] Therefore, if $g_l^+\in (H_\infty^+)^{1\times n}$ is the left inverse of $g^+$ over $H_\infty^+$, using the notation \[ \Delta(\widetilde{G}^+):=[\Delta_{\cdot;1}(\widetilde{G}^+), \Delta_{\cdot;2}(\widetilde{G}^+),\ldots, \Delta_{\cdot;n}(\widetilde{G}^+)], \] we have \eq{4.27} g_l^+\cdot \Delta(\widetilde{G}^+)=\left(\sum_{j=1}^{n-1}r_j^{-1}\right)rg_l^+g^+=  \left(\sum_{j=1}^{n-1}r_j^{-1}\right)r\in H_\infty^+\cap\mathcal{GR}.\en Let now, in the notation of \eqref{4.22},
\[ g_l^+= [(g_l^+)_1, (g_l^+)_2, \ldots, (g_l^+)_n]^T\in (H_\infty^+)^{n\times 1},\quad  h_k^+:=(-1)^{k+1}(g_l^+)_k, \quad k=1,2,\ldots, n,\]  and \[ h^+=[h_1^+, h_2^+, \ldots h_n^+]^T\in (H_\infty^+)^{1\times n}. \] Let moreover \eq{4.31} \widetilde{M}^+=[h^+ \widetilde{g}_1^+\ \widetilde{g}_2^+\ \ldots \ \widetilde{g}_{n-1}^+] = [ h^+\, | \, \widetilde{G}^+]. \en Then $\widetilde{M}^+\in (H_\infty^+)^{n\times n}$ and \[ \det \widetilde{M}^+=g_l^+\Delta(\widetilde{G}_+)\in H_\infty^+\cap\mathcal{GR}\] by \eqref{4.27}. Following Theorem~3.4 in \cite{CaLeS} or Lemma~2.1 in \cite{LS}, we see that there exist $R\in \mathcal{GR}^{n\times n}$ and $M^+\in\mathcal{G}(H_\infty^+)^{n\times n}$ such that $R^{-1}\widetilde{M}^+= M^+$.  From here and \eqref{4.31}, \eq{4.34} R^{-1}\widetilde{G}^+=G^+, \en where $G^+$ is left invertible over $H_\infty^+$, its left inverse equals $(M^+)^{-1}$  with the first row deleted.

It follows from \eqref{4.20} and \eqref{4.34} that \eqref{4.19} holds, with $F=G^+$. \end{proof}
\bigskip

\begin{rem} {\rm The obvious analogues of Theorems~\ref{aug93} and \ref{aug910} for right invertible matrices over ${\mathcal
A}$ are also valid. We will not explicitly state these analogues, but use them as
needed in the sequel.}
\end{rem}

\section{Fredholmness of Toeplitz operators and factorization}\label{s:Fred}

Let $L_p(\R)$, $1<p \leq \infty$, be the standard Lebesgue spaces of functions on the real
line $\R$ with respect to the Lebesgue measure, while $H_p^\pm$ denote the Hardy spaces
$H^p(\C^\pm)$ in the open  upper (resp. lower) halfplane $\C^+$ (resp. $\C^-$). For $1<p<\infty$, $H_p^\pm$
consists of all functions $f$ holomorphic in $\C^\pm$ for which
\[ \sup_{\pm y>0} \int_{-\infty}^\infty |f(x+iy)|^p dx <\infty . \] This definition is standard in many sources, see, e.g.,
\cite{Dur,GKru921,Koosis98,Wid60} for basics on $H_p^\pm$ and associated singular integral operators.

For $p\in ]1,\infty [$, the space $L_p(\R)$ splits into the direct sum of $H_p^+$ and $H_p^-$: $L_p(\R)=H_p^+\dot{+} H_p^-$.
We denote by $P^\pm$ the projection of $L_p(\R)$ onto $H^{\pm}_p$ parallel to $H^\mp_p$. We will also need the  modified projections $\widetilde{P}^\pm$, acting from $L_\infty(\R)$  into ${\mathcal  L}_p^\pm:=(\xi \pm i)H_p^\pm$,
$p\in ]1,\infty [$, by
\eq{II.1} \widetilde{P}^\pm \phi=(\xi+i)P^\pm \left(\frac{\phi}{\xi+i}\right).
\en

Toeplitz operators with matrix \emph{symbol} $G\in (L_\infty(\R))^{n\times n}$ are defined as follows:
\eq{2.11}
    T_G:(H_p^+)^n\longrightarrow (H_p^+)^n, \hspace{0.4cm}T_G\phi^+=P^+G\phi^+
    \hspace{0.4cm}(p\in ]1,+\infty[).
\en
There is a close relation between properties of Toeplitz operators and factorization of their symbols. Thus, we remind now the basic definitions and properties concerning the latter.

Given $p\in (1,\infty)$, an $L_p$-{\em factorization} of a function
 $G\in (L_\infty(\R))^{n\times n}$ is defined as a representation
\begin{equation}\label{2.5}
    G=G_-DG_+,
\end{equation}
where $D$ is a diagonal rational matrix of the form
\begin{equation}\label{2.6}
    D={\rm diag}\, (r^{k_j})_{j=1,2,\dots,n}\,, \;\; k_j\in\Z \;\;
    \text{for all} \;j=1,2,\dots n,
\end{equation}
\begin{equation}\label{2.4}
    r(\xi)=\frac{\xi-i}{\xi+i}\;, \hspace{0.2cm}\text{for} \hspace{0.2cm}
    \xi\in\R,
\end{equation}
and the factors $G_\pm$ are such that, for
\begin{equation}\label{A2.4}
    p'=\frac{p}{p-1}\,,\;\; \lambda_\pm(\xi)=\xi\pm i\;
    (\xi\in\R),
\end{equation}
we have
\eq{2.7}
\lambda_+^{-1}G_+^{-1}\in (H_p^+)^{n\times n}, \quad \lambda_+^{-1}G_+\in (H_{p'}^+)^{n\times n}
\en
\eq{2.8}
    \lambda_-^{-1}G_-\in (H_p^-)^{n\times n}\;, \hspace{0.3cm} \lambda_-^{-1}G_-^{-1}\in
    (H_{p'}^-)^{n\times n}. 
\en Under conditions \eqref{2.7}, \eqref{2.8},  $G_-P^+G_-^{-1}I$
can be considered as a closable operator on $(L_p(\R))^n$ defined on a dense linear set $\lambda_{+}^{-1}G_+{\mathcal R}^n$. If, in addition,
\eq{2.9} G_-P^+G_-^{-1}I \text{ is bounded in the metric of }  (L_p(\R))^n \en
(and therefore extends onto $(L_p(\R))^n$ by continuity), we say that \eqref{2.5} is a {\em Wiener-Hopf} ({\em WH}) $p$-factorization of $G$.

For each $p$, the diagonal middle factor in (\ref{2.5}) is unique up to the order of its
diagonal elements, and the integers $k_j$
are called the \emph{partial indices} of $G$, its sum  ${\rm Ind}_p\,(G)$ being the
(total) $p$-{\em index} of $G$. In the case of a scalar symbol possessing a {\em WH} $p$-factorization,
the partial and the total indices coincide and will be simply called  the $p$-\emph{index} of $G$.

The factorization (\ref{2.5}) is said to be {\em bounded} if
\begin{equation}\label{A2.12}
    G_+\in \mathcal{G}(H_\infty^+)^{n\times n}, \quad     G_-\in \mathcal{G}(H_\infty^-)^{n\times n}.
\end{equation}
Clearly, a bounded factorization is a {\em WH} $p$-factorization\, for all
$p\in ]1,+\infty[$.

Any matrix function in $\mathcal {GR}^{n\times n}$ admits a
factorization \eqref{2.5} with $G_\pm\in \mathcal {G}(\mathcal{R}^\pm)^{n\times n}$, where ${\mathcal R}^\pm:= {\mathcal R}\cap H_\infty^\pm$ is the subalgebra of $\mathcal R$ consisting of all rational functions without poles in $\C^\pm\cup\{\infty\}$.

In particular, every scalar function in $\mathcal {GR}$ is the product of
functions in $\mathcal {GR}^+$, $\mathcal {GR}^-$, and some integer
power of the function $r$ defined by \eqref{2.4}. Thus, without loss of
generality condition $s\in\mathcal {GR}$ in \eqref{aug81} may be
substituted by $s=s_\mp r^j$, where $s_\mp\in\mathcal {GR}^\mp$
and $j\in\Z$.

The relation between Fredholm properties of $T_G$ and factorization \eqref{2.5} is well known; see e.g. \cite[Theorem 5.2]{MP86}.
For convenience of reference, we give the precise statement here (as it was done also in \cite{CDR}).
\begin{thm} \label{feb31}
Let $G\in (L_\infty(\R))^{n\times n}$, $p\in ]1,+\infty[$.  Then $T_G$ is
Fredholm on $(H^+_p)^n$ if and only if $G$ admits a WH
$p$-factorization.
\end{thm}

The partial indices are related to the dimension of the kernel and the
cokernel of $T_G$ by
\begin{equation}\label{A2.10}
    \dim\ker T_G=\sum_{k_j\leq 0}|k_j|\,, \hspace{0.5cm}
    \dim\coker T_G=\sum_{k_j\geq 0} k_j.
\end{equation}
Thus, the index of $T_G$, ${\rm Ind}\, T_G$, is given by (see Theorem \ref{feb31})
$$ {\rm Ind}\, T_G:=
 {\rm dim}\, ({\rm Ker}\, T_G) - {\rm dim}\, (\coker T_G)
=-{\rm Ind}_p \, G. $$ We see thus that the existence of a
\emph{canonical p-factorization\,} for $G$ is particularly interesting,
since it is equivalent to invertibility for $T_G$. Moreover, the inverse
operator can then be defined in terms of $G_\pm$ by
\begin{equation}\label{A2.11}
    T_G^{-1}=G_+^{-1}P^+G_-^{-1}I.
\end{equation}

\section{One sided invertibility and Fredholmness of Toeplitz operators}\label{s:osift}
 In this section we show that one-sided invertibility over the algebras $H^\pm_\infty+C$ or $H^\pm_\infty$ of certain submatrices of the $n\times n$ matrix function $G$ implies that the Toeplitz operators $T_G$ and $T_{\det G}$ are at least nearly, and in some cases strictly, Fredholm equivalent. In particular, in the latter case $T_G$ possesses Coburn's property.

Given $\Phi^{\pm}\in (H^\pm_\infty+C)^{n \times (n-1)}$, $\Psi^{\pm}\in
(H^\pm_\infty+C)^{(n-1) \times n}$ such that $\Psi^{\pm}\Phi^{\pm}=I_{n-1}$,
let moreover $\Phi^\pm_e, \Psi^\pm_e$ be defined by
\begin{equation}\label{II.2}
\Phi^\pm_e=\left[ \Phi^\pm \ \ \ N^\pm\right], \qquad \Psi^\pm_e=\left[\begin{array}{c} \Psi^{\pm} \\ \widetilde{N}^\pm
\end{array}\right],
\end{equation}
where
\begin{eqnarray} \label{II.3'}
N^\pm&=&\left[\begin{array}{c} \Delta_{\cdot,1} (\Psi^\pm) \\ -\Delta_{\cdot, 2} (\Psi^\pm)\\ \vdots
\\ (-1)^{n-1} \Delta_{\cdot, n} (\Psi^\pm) \end{array} \right]\in (H_\infty^\pm+C)^{n\times 1}, \\ \label{II.3}
\widetilde{N}^\pm&=& \left[ \Delta_{1,\cdot} (\Phi^\pm),  -\Delta_{2, \cdot} (\Phi^\pm), \ldots ,
(-1)^{n-1} \Delta_{n, \cdot} (\Phi^\pm) \right]\in (H^\pm_\infty+C)^{1\times m}.
\end{eqnarray}

\begin{thm}\label{aug112}
Let $G\in (L_\infty (\R))^{n \times n}$, and let $\Psi$ be an
$(n-1)\times n$ submatrix of $G$ obtained by omitting one row in $G$.

${\rm (a)}$ If $\Psi\in (H_\infty^++C)^{(n-1)\times n}$, and if $\Psi$ is right invertible over $H_\infty^++C$, then
$T_G$ is nearly Fredholm equivalent to $T_{{\rm det}\, G}$, for every fixed $p\in ]1,\infty[$.

${\rm (b)}$ If moreover  $\Psi\in (H_\infty^+)^{(n-1)\times n}$, and if
 $\Psi$ is right invertible over $H_\infty^+$,
then, for any fixed $p\in  ]1,\infty[$, $\ker T_G=\{0\}$ or $\ker T_G^*=\{0\}$, and $T_G$ is strictly Fredholm equivalent to $T_{\det G}$. In particular, $T_G$ is one- or two- sided invertible simultaneously with $T_{\det G}$.

 ${\rm (c)}$ If, in the setting of ${\rm (b)}$, in addition $\Ind T_{\det G}\geq
0$ and the omitted row $\widehat{G}_n$ of $G$ is its last one, then
 a WH
$p$-factorization of $G$ is given by {\em\eqref{2.5}} with
\begin{equation}\label{II.5}
G_-=\left[\begin{array}{cc} I_{n-1} & 0 \\ 0 & \gamma_-\end{array} \right]
\ \left[\begin{array}{cc} I_{n-1} & 0 \\ \widetilde{P}^-(\widehat{G}_n\Phi^+\gamma_-^{-1})  & 1\end{array} \right],
\end{equation}
\begin{equation}\label{II.6}
D= \left[\begin{array}{cc} I_{n-1} & 0 \\ 0 & r^k\end{array} \right];
\end{equation}
\begin{equation}\label{II.7}
G_+= \left[\begin{array}{cc} I_{n-1} & 0 \\ \widetilde{P}^+(\widehat{G}_n\Phi^+\gamma_-^{-1})\cdot r^{-k}  & 1\end{array}
\right]
\ \left[\begin{array}{cc} I_{n-1} & 0 \\ 0 & (-1)^{n-1}\gamma_+\end{array} \right] \ \cdot \ \Psi^+_e.
\end{equation}
Here \begin{equation}\label{II.4} \det
G=\gamma_-r^k\gamma_+
\end{equation} is a WH $p$-factorization of $\det G$ and $\Phi^+$ is a right inverse of $\Psi$.
\end{thm}
Note that in \eqref{II.4} $k\leq 0$ since it is opposite to
$\Ind T_{\det G}$. This condition is essential for the statement (c) to
be valid,  while of course omitting the $n$-th
row (as opposed to some other row) is just to simplify the notation.

\begin{proof}  (a) Since $\Psi$ is right invertible over $H_\infty^++C$, by Theorem~\ref{aug111} (taking (\ref{I.24}) into account) $\Phi^+_e\in {\mathcal
G}(H_\infty^++C)^{n \times n}$ and --- see (\ref{I.19}), where $G$ takes
the place of $\Psi^+_e$ --- we have
\begin{equation} \label{II.8}
G=\widetilde{G}(\Phi^+_e)^{-1}, \text{ where } \widetilde{G}=
\left[\begin{array}{cc} I_{n-1} & 0_{(n-1)\times 1} \\ \widehat{G}_n \Phi^+ & (-1)^{n-1}{\rm det}\, G
\end{array}\right].
\end{equation}
By \cite[Theorem 5.5]{LS} (or rather its version for the right
factorization) $G$ is $WH$ $p$-factorable only simultaneously with
$\widetilde{G}$. In its turn, this matrix function  is block triangular, with
one of the blocks (the identity matrix) obviously $WH$ $p$-factorable.
According to \cite[Corollary 4.1]{LS}, $\widetilde{G}$  itself is $WH$
$p$-factorable only simultaneously with its other diagonal block, that
is, the function $\det G$. In the language of Toeplitz operators this
means that $T_G$ and $T_{\det G}$ are nearly Fredholm equivalent.

(b) If moreover  $\Psi^+\in
(H_\infty^+)^{(n-1)\times n}$ is right invertible over $H_\infty^+$, then
$(\Phi_e^+)^{-1}$ (which is equal to $\Psi^+_e$ by Lemma
\ref{aug111}) is invertible in $(H_\infty^+)^{n \times n}$.
Consequently, from \eqref{II.8}, $\ker T_{G}=\{0\}\Leftrightarrow \ker T_{\widetilde{G}}=\{0\}\Leftrightarrow \ker T_{\det G}=\{0\}$, and analogously $\ker T^*_{G}=\ker T_{G^*}=\{0\}\Leftrightarrow \ker T^*_{\det G}=\ker T_{\overline{\det G}}=\{0\}$. Since, by Coburn's property, $\ker T_{\det G}$ or $\ker T^*_{\det G}$ is $\{0\}$, the same is true regarding $\ker T_G$ and $\ker T^*_{G}$. On the other hand, since in \eqref{II.8} we have $(\Phi_e^+)^{\pm 1}\in (H_\infty^+)^{n\times n}$, not only $G$ and $\widetilde{G}$ admit WH $p$-factorizations along with $\det G$, but also their partial indices coincide.
Due to the triangular structure of $\widetilde{G}$, the set of its partial
indices is majorized  by the set of the indices of its diagonal entries
\cite{Spit801}, see also \cite[Theorem 4.7]{LS}. Without going into
details of the majorization relation and its properties, we note here
only the following pertinent piece of information:  since the indices of
the diagonal entries of $\widetilde{G}$ are $0,\ldots, 0$ ($n-1$ times)
and $k$, all its partial indices are of the same sign as $k$, and their
sum amounts to $k$. According to \eqref{A2.10}, the defect numbers
of $T_G$ are the same as those of $T_{\det G}$. In particular, one of
them is zero, and the other coincides in absolute value with $\Ind
T_{\det G}$. This guarantees one sided invertibility (which becomes
two sided if and only if $k=0$).

(c) Under the condition $k\leq 0$, the matrix functions $G_\pm$ defined by
\eqref{II.5}-\eqref{II.7} satisfy \eqref{2.7}, \eqref{2.8}. Since $G$ is $WH$
$p$-factorable and \eqref{2.5} holds, condition \eqref{2.9} is satisfied
automatically \cite[Theorem 3.8]{LS}, and \eqref{II.5}--\eqref{II.7} deliver
the desired $WH$ $p$-factorization. In particular, the partial indices of
$G$ are $0,\ldots,0$ ($n-1$ times) and $k$, so they coincide with the
indices of the diagonal entries of $\widetilde{G}$.
\end{proof}

The next result is a dual version of Theorem \ref{aug112}.
\begin{thm}\label{th:aug113}
Let $G\in (L_\infty (\R))^{n \times n}$, and let $\Phi$ be an $n\times
(n-1)$ submatrix of $G$ obtained by omitting one column in $G$ (it
will be assumed that the $n$th column is omitted, essentially without
loss of generality).

${\rm (a)}$  If $\Phi\in (H_\infty^-+C)^{n\times (n-1)}$, and if $\Phi$ is left invertible over $H_\infty^-+C$, then
$T_G$ and $T_{\det G}$ are nearly Fredholm equivalent, for every fixed $p\in ]1,\infty[$.

${\rm (b)}$ If moreover $\Phi\in (H^-_\infty)^{n\times (n-1)}$ and $\Phi$ is left invertible over $H^-_\infty$, then, for any fixed $p\in (1,\infty)$, $\ker T_G=\{0\}$ or $\ker T^*_G=\{0\}$, and the operator $T_G$ is strictly Fredholm equivalent to $T_{\det G}$. In particular, $T_G$ is invertible if and only if so is $T_{\det G}$.

${\rm (c)}$ If, in the setting of ${\rm (b)}$, in addition $\Ind\det G\leq 0$ and the omitted column of $G$ is $\widehat{G}_n$, its last one, then a $WH$ $p$-factorization of $G$ is given by {\em\eqref{2.5}} with \[ G_-=\Phi_e^-\left[\begin{matrix} I_{n-1} &  r^{-k}\widetilde{P}_-(\gamma_+^{-1}\Psi_-\widehat{G}_n)\\ 0 & (-1)^{n-1}\gamma_-\end{matrix}\right], \quad D=\left[\begin{matrix} I_{n-1} & 0 \\ 0 & r^k\end{matrix}\right], \] \[ G_+= \left[\begin{matrix} I_{n-1} &  \gamma_+\widetilde{P}_+(\gamma_+^{-1}\Psi_-\widehat{G}_n) \\ 0 & \gamma_+\end{matrix}\right]. \] Here $\det G=\gamma_-r^k\gamma_+$ is a $WH$ $p$-factorization of $\det G$, $\Psi_-$ is a left inverse of $\Phi$, and $\Phi_e^-=(\Psi_e^-)^{-1}$ is given by \eqref{II.2}--\eqref{II.3}.
\end{thm}

Of course, formulas similar to those given in Theorem~\ref{th:aug113}(c) hold when the removed column is not the last one.

In the previous results we have used the one sided invertibility of a submatrix of $G$ to study the Fredholmness, and other associated properties, of the Toeplitz operator $T_G$. Now we turn to the study of the same properties of $T_G$ based on one sided invertibility of a solution to a Riemann-Hilbert problem with coefficient $G$.

\begin{thm}\label{aug117} Let $G\in (L_\infty(\R))^{n \times n}$, and let
\begin{equation}\label{II.10} G\Phi^+=\Phi^-, \quad \Phi^{\pm}\in (H_{\infty}^{\pm}+C)^{n \times (n-1)},
\end{equation}
where $\Phi^\pm$ are left invertible over $H_\infty^\pm+C$. Then:
\begin{itemize}
\item[(i)] $T_G$ is nearly Fredholm equivalent to $T_{\det G}$;
\item[(ii)] If moreover
    $\Phi^{\pm}$ are left invertible over $H_\infty^\pm$, with left
    inverses $\Psi^\pm \in (H_\infty^\pm)^{(n-1)\times n}$ for
    $\Phi^\pm$, respectively, then $\ker T_G=\{0\}$ or $\ker T^*_G=\{0\}$, and $T_G$ is strictly Fredholm equivalent to $T_{\det G}$. In particular, $T_G$ is invertible if and only if $T_{\det G}$ is invertible.

\item[(iii)]   Assuming that
\begin{equation}\label{II.10A}
\det G=\gamma_- r^k \gamma_+ \qquad \text{ with } k\geq 0
\end{equation}
is a WH $p$-factorization for ${\rm det}\, G$, a WH $p$-factorization for $G$ is given by {\em\eqref{2.5}} with
\begin{equation}\label{II.10B}
G_-=\Phi_e^- \ \cdot \ \left[\begin{array}{cc} I_{n-1} & 0 \\ 0 & \gamma_- \end{array}\right] \
\left[\begin{array}{cc} I_{n-1} & \alpha_- \\0 & 1 \end{array}\right],
\end{equation}
\begin{equation}\label{II.10C}
D=\left[\begin{array}{cc} I_{n-1} & 0_{(n-1)\times 1}  \\ 0_{1\times (n-1)}  & r^k \end{array}\right]
\end{equation}
\begin{equation}\label{II.10D}
G_+=
\left[\begin{array}{cc} I_{n-1} & \alpha_+ \\ 0_{1 \times (n-1)} & 1 \end{array}\right] \ \left[\begin{array}{cc}
I_{n-1} & 0 \\ 0 & \gamma_+
\end{array}\right] \ \cdot \ \Psi_e^+,
\end{equation}
where $\Phi^-_e$, $\Psi^+_e$ are given by {\em (\ref{II.2})--(\ref{II.3})},
\begin{equation}\label{II.10E}
\alpha_+= \widetilde{P}^+(Q) \in ({\mathcal  L}_p^+)^{(n-1)\times 1},
\end{equation}
\begin{equation}\label{II.10F}
\alpha_-= r^{-k} \widetilde{P}^-(Q) \in ({\mathcal  L}_p^-)^{(n-1)\times 1},
\end{equation}
and where
\begin{equation}\label{II.15}
Q:=\Psi^-GN^+\in (L_\infty (\R))^{(n-1)\times 1},
\end{equation}
with $N^+$ as in {\em (\ref{II.3'})}.
\end{itemize}
\end{thm}

\begin{proof} (i) Let $\Phi_e^\pm$, $\Psi_e^\pm$ be defined as in (\ref{II.2}),
(\ref{II.3}), where $\Psi^{\pm}\in  (H_{\infty}^{\pm}+C)^{(n-1) \times n}$
is a left inverse of $\Phi^\pm$ over $H_{\infty}^{\pm}$. From
Theorem~\ref{aug111} it follows that $\Psi_e^{\pm}\in {\mathcal  G}(H_{\infty}^{\pm}+C)^{n \times n}$ and
$(\Psi_e^{\pm})^{-1}=\Phi^\pm_e$.

Defining
\begin{equation}\label{II.11}
G_0=\Psi^-_eG\Phi^+_e,
\end{equation}
we can rewrite (\ref{II.10}) as
\begin{equation}\label{II.12}
G_0\Psi^+_e\Phi^+=\Psi^-_e\Phi^-.
\end{equation}
On the other hand, it also follows from Theorem~\ref{aug111} (see (\ref{I.22}) or (\ref{I.24}), taking
(\ref{I.16}) into account) that
\begin{equation}\label{II.13}
\Psi_e^\pm\Phi^\pm =\left[\begin{array}{c} I_{n-1} \\ 0_{1 \times (n-1)} \end{array}\right],
\end{equation}
therefore (\ref{II.12}) implies that $G_0$ has the form
\begin{equation}\label{II.14}
G_0=\left[\begin{array}{cc} I_{n-1} & Q \\ 0_{1\times (n-1)} & {\rm det}\, G \end{array}\right].
\end{equation}
In particular, ${\rm det}\, G={\rm det}\, G_0$.

From (\ref{II.11}) it follows according to \cite[Theorem 5.5]{LS} that $T_G$ is Fredholm if and only if $T_{G_0}$ is Fredholm, and this in turn
is equivalent to $T_{{\rm det}\, G}$ being Fredholm (by (\ref{II.14})).

 (ii) If  $\Phi^\pm\in (H_\infty^\pm)^{n
\times (n-1)}$ and $\Phi^\pm$ is left invertible over
$H_\infty^\pm$, with a left inverse $\Psi^\pm$, then
$$ \Psi^-_e\in {\mathcal  G}\left((H_\infty^-)^{n \times n}\right), \qquad
 \Phi^+_e\in {\mathcal  G}\left((H_\infty^+)^{n \times n}\right), $$
and it follows that $T_G$ is strictly Fredholm equivalent to $T_{\det G}$ and that $\ker T_G=\{0\}$ or $\ker T_G^*=\{0\}$ (see a
similar reasoning in the proof of Theorem~\ref{aug112}).

(iii)  The formulas for $G_\pm$ and $D$ follow from
$$ G=\Phi^-_eG_0\Psi^+_e, $$
together with (\ref{II.14}) and (\ref{II.10A}).\end{proof}

\section{Special cases}\label{s:Final}

Let $G\in L_\infty^{n\times n}$ with all rows but one having elements in  $M_\infty^+$ (the case of all columns but one having elements in $M^-_\infty$ can be treated analogously). Assume for simplicity that
\eq{6.1} G=\left[\begin{matrix} \Psi \\ g_n\end{matrix}\right] \text{ with } \Psi\in (M_\infty^+)^{(n-1)\times n}, \ g_n\in L_\infty^{1\times n}. \en Then the following results hold. \begin{thm}\label{th:6.1} {\em (i)} If $G$ is unitary with constant determinant and $g_n^T\in MCT_n^-$, then $T_G$ is Fredholm for all $p\in (1,\infty)$. If moreover $\Psi\in (H_\infty^+)^{(n-1)\times n}$ and $g_n^T\in HCT_n^-$, then $T_G$ is invertible.

{\em (ii)} If $G$ is (complex) orthogonal with constant determinant and $g_n^T\in MCT_n^+$, then $T_G$ is Fredholm for all $p\in (1,\infty)$. If moreover $\Psi\in (H_\infty^+)^{(n-1)\times n}$ and $g_n^T\in HCT_n^+$, then $T_G$ is invertible.

{\em (iii)} If one of the $(n-1)\times (n-1)$ minors of $\Psi$ is invertible in $M^+_\infty$, then $T_G$ is nearly Fredholm equivalent to $T_{\det G}$. If the above mentioned minor is in fact invertible in $H^+_\infty$, then $T_G$ is strictly Fredholm equivalent to $T_{\det G}$ and $\ker T_G=\{0\}$ or $\ker T_G^*=\{0\}$.  \end{thm}

Of course $\abs{\det G}=1$ in case (i) and $\det G=\pm 1$ in case (ii).

\begin{proof} (i) Observe that the $k$-th entry $g_{nk}$ of $g_n$ coincides with $(-1)^{n+k}\det G\overline{\Delta_{.,k}(\Psi)}$, $k=1,\ldots,n$. Thus, condition $g_n^T\in MCT_n^-$ can be rewritten equivalently as \eq{delta} (\Delta_{.,k}(\Psi))_{k=1,\ldots,n}\in MCT^+_n.\en By the right invertibility analogue of Theorem~\ref{aug93}, it follows that $\Psi$ is right invertible over $M_\infty^+$, and Theorem~\ref{aug112}(a) implies that $T_G$ is Fredholm. The second part of (i) follows analogously from Theorem~\ref{aug112}(b).

(ii) If $G$ is orthogonal, then $g_{nk}=(-1)^{n+k}\det G\Delta_{.,k}(\Psi)$, so that now $g_n^T\in MCT_n^+$ can be rewritten as \eqref{delta}. The rest of the proof goes as in (i).

(iii) The invertibility of any $(n-1)\times (n-1)$ minor of $\Psi\in (M^+_\infty)^{(n-1)\times n}$ implies \eqref{delta}. So, Theorems~\ref{aug93} and \ref{aug112} again do the job. \end{proof}

\begin{rem}In the case of orthogonal \eqref{6.1} we automatically have $g_n\in (M^+_\infty)^{1\times n}$, so that the relation $GG^T=I$ immediately provides the right inverse of $\Psi$ over $M^+_\infty$. It can then be used in factorization formulas \eqref{II.5}--\eqref{II.7} of Theorem~\ref{aug112}(c).

A factorization of unitary matrices $G$ with $\det G=1, \Psi\in (H_\infty^+)^{(n-1)\times n}$ and $g_n^T\in HCT_n^-$ as in Theorem~\ref{th:6.1}, was by different methods considered earlier in \cite{EJL98}.
\end{rem}

We will now show that the one sided invertibility requirement in part (a) of Theorems~\ref{aug112}, \ref{th:aug113} can be lifted if the submatrix in question is continuous. First we will dispose of the case when it is rational.

\begin{lem}\label{l:rat} Let $G$ be of the form{\em\eqref{6.1}} with $\Psi\in{\mathcal R}^{(n-1)\times n}$. Then $T_G$ is nearly Fredholm equivalent to $T_{\det G}$. \end{lem}

\begin{proof}If the determinants $\Delta_{.,k}(\Psi)$, $k=1,\ldots, n$ (which are rational functions in $\mathcal R$) have at least one common zero in $\dot{\R}$, then $\det G$ has the same zero and thus neither $T_{\det G}$ nor $T_G$ is Fredholm.

Suppose now there are no common zeros of $\Delta_{.,k}(\Psi)$ in $\dot{\R}$. Since there are at most finitely many such zeros in $\C^\pm$,  then \eqref{delta} holds again. By Theorems~\ref{aug93}, $\Psi$ is right invertible over $M^+_\infty$. The statement now follows from Theorem~\ref{aug112}.
 \end{proof}

\begin{thm}\label{th:cont} Let $G\in L_\infty^{n\times n}$ be such that all its elements except maybe for those located in one row or one column are continuous on $\dot{\R}$. Then $T_G$ is nearly Fredholm equivalent to $T_{\det G}$.  \end{thm}
\begin{proof} Without loss of generality, $G$ is of the form \eqref{6.1} with $\Psi\in C^{(n-1)\times n}$.

{\sl Necessity.} Suppose $T_G$ is Fredholm. Then $\det G$ is invertible in $L_\infty$. Expanding $\det G$ across the last row, represent it as \[
\det G = \sum_{j=1}^n f_j g_{n,j}, \] where the cofactors $f_j$ are continuous due to the continuity of $\Psi$. Let us approximate $\Psi$ by a rational matrix function $\widetilde{\Psi}$ so closely that the Toeplitz operator with the modified symbol $G_1=\left[\begin{matrix}\widetilde{\Psi} \\ g_n\end{matrix}\right]$ remains Fredholm. In particular, $\det G_1=\sum_{j=1}^n\widetilde{f}_jg_{n,j}$ is still invertible.

Now let \[ \widetilde{g}_{n,j}=g_{n,j}\det G/\det G_1, \ \widetilde{g_n}=[\widetilde{g}_{n,1} \ldots \widetilde{g}_{n,n} ] \text{ and } \widetilde{G}= \left[\begin{matrix}\widetilde{\Psi} \\ \widetilde{g}_n\end{matrix}\right] . \] The matrix function $\widetilde{G}$ can be made arbitrarily close to $G$, so that we may suppose $T_{\widetilde{G}}$ to be Fredholm. By Lemma~\ref{l:rat}, the operator $T_{\det\widetilde{G}}$ is Fredholm. It remains to observe that \[ \det \widetilde{G}= \sum_{j=1}^n \widetilde{f}_j\widetilde{g}_{n,j}=\det G. \]

{\sl Sufficiency.} Along with $T_G$, let us consider $T_{\adj G}$, where $\adj G$ stands for the transposed matrix of the cofactors of $G$.
Recall that \eq{adj} G\adj G = \adj G\, G = (\det G) I_n, \en
and let $I_+$, $I_+^n$ denote the identity operators on $H_p^+$,
$(H_p^+)^n$, respectively.
Since the first $n-1$ rows of $G$ and the last column of $\adj G$ are  continuous
on  ${\R}\cup \{\infty\}$, the operator $$k_\ell:=T_{\adj G}T_G-T_{G\adj G}$$ is compact (Corolary 3.5 in \cite{MP86}).
Taking (\ref{adj}) into account, we conclude that
$$ T_{\adj G}T_G=(\det G)I^n_+  + k_\ell $$
is Fredholm and therefore $T_G$ has a left regularizer
(that is, a left inverse modulo the ideal of compact operators).

To show that $T_G$ has also a right regularizer --- and therefore
$T_G$ is Fredholm --- we consider $T_GT_{\adj G}$. In this case, the difference
$T_GT_{\adj G}-T_{G\cdot {\adj G}}$ may not be compact, so we have to use different (and somewhat more involved) arguments. Let $[T_{ij}]$, $(i,j\in \{1,2\})$,
be the block representation of the operator
$$ T_GT_{\adj G}-T_{G\cdot {\adj G}}=T_GT_{\adj G}- (\det G)I^n_+, $$
corresponding to the decomposition $(H_p^+)^n =(H_p^+)^{n-1}
\oplus H_p^+. $ The operators $T_{11}$, $T_{12}$, and $T_{22}$ are compact
(by \cite[Corollary 7.5]{MP86}), and we can write
\begin{eqnarray*}
 T_GT_{\adj G}& = & (\det G)I_+^n + \left[\begin{array}{cc} T_{11} & T_{12}\\
 T_{21} & T_{22} \end{array}\right]
 \\ & = & \left[\begin{array}{cc} (\det G)I_+  & 0\\
 T_{21} &  (\det G)I_+  \end{array}\right]+ \left[\begin{array}{cc} T_{11} & T_{12}\\
 0 & T_{22} \end{array}\right].
 \end{eqnarray*}

Thus $T_G T_{\adj G}$ is a compact perturbation of a block triangular operator
which is Fredholm since its diagonal elements are Fredholm
(see, e.g., Corollary~1.3 in \cite{LS}). Thus $T_G T_{\adj G}$ is Fredholm which implies that $T_G$ has a right regularizer as well. \end{proof}

Note that some relations between semi-Fredholmness of $T_G$ and $T_{\det G}$ in the setting of Theorem~\ref{th:cont} can be extracted from Markus-Feldman results (\cite{MF80}, see also \cite[Chapter 1]{Kru87}) on the one sided invertibility of matrices over some non-commutative ring but with entries from different rows or columns pairwise commuting.

\section{Toeplitz operators with almost periodic symbols}\label{s:ap}

The difference between the scalar and matrix settings becomes even more profound for Toeplitz operators with almost periodic symbols. To define the latter, first we introduce $APP$, the (non-closed) algebra of {\em almost periodic polynomials}, that is, linear combinations of the functions
$e_\lambda(t):=e^{i\lambda t}$, $\lambda\in\R$. The Banach algebra $AP$ of almost periodic functions by definition is the closure
of $APP$ in $L_\infty(\R)$. We will also need $APW$, the closure of $APP$ in the {\em Wiener norm}
\[ \norm{\sum c_je_{\lambda_j}}=\sum\abs{c_j}, \] with no repetitions in the set $\{\lambda_j\}$.

Let further $AP^\pm$ ($APW^\pm$) denote the closure in $AP$ (respectively, $APW$) of all $f=\sum c_je_{\lambda_j}\in APP$ with $\pm\lambda_j\geq 0$.
Note that $AP^\pm$ (respectively, $APW^\pm$) consist of all functions $f\in
AP$ (respectively, $APW$) that admit holomorphic continuation into
$\C^{\pm}$. Of course, $AP^\pm$ and $APW^\pm$ are unital Banach
subalgebras of $AP$ and $APW$, respectively.

For any $f\in\mathcal{G}AP$ there exists a
unique $\kappa\in\R$ such that a continuous branch of
$\log(e_{-\kappa}f)$ lies in $AP$. This $\kappa$ is called the {\em
mean motion} of $f$, and is sometimes denoted $\kappa(f)$.

Operators $T_f$ with scalar $f\in AP$ were treated by Coburn-Douglas \cite{CD} and Gohberg-Feldman \cite{GF68}, and the situation with them is as follows: the operator $T_f$ is semi-Fredholm  if $f\in {\mathcal G}AP$ and has a non-closed range otherwise. Moreover, for $f\in {\mathcal G}AP$ with $\kappa(f)=0$, $T_f$ is invertible while in the case of non-zero $\kappa(f)$ one of its defect numbers is infinite. In particular, $T_f$ is Fredholm only if it is invertible.

The latter property persists for matrix $AP$ symbols, see \cite[Chapter 18]{BKS1}. However, it is no longer true that the invertibility of $G\in AP^{n\times n}$, or even $APW^{n\times n}$, implies the semi Fredholmness of $T_G$. Moreover, there exist \eq{triang} G=\begin{bmatrix} e_{-\lambda} & 0 \\ f & e_\lambda \end{bmatrix} \en (so that $\det G\equiv 1$) with $\lambda>0$ and $f\in APP$  for which the range of $T_G$ is not closed \cite{SY}.

To describe the situation further, we introduce the notion of $AP$ and $APW$ factorization.

Representation  \eqref{2.5} in which $G_\pm$ satisfy \eq{apf}
G_+\in \mathcal{G}(AP^+)^{n\times n}, \quad  G_-\in
\mathcal{G}(AP^-)^{n\times n} \en and the diagonal elements of $D$ have the
form $e_{\mu_j}$, as opposed to \eqref{2.6}, is called a (right) {\em
$AP$ factorization} of $G$.  An $AP$ factorization of $G$ is
by definition its $APW$ factorization if conditions
\eqref{apf} are strengthened to
\[ G_+\in \mathcal{G}(APW^+)^{n\times n}, \quad  G_-\in
\mathcal{G}(APW^-)^{n\times n}.
\]  The real parameters $\mu_j$ are defined uniquely, provided that an $AP$ (or $APW$) factorization of $G$ exists,
and are called its {\em partial $AP$ indices}. Of course, a canonical
(that is, satisfying $\mu_1=\ldots=\mu_n=0$) $AP$ factorization of $G$
is at the same time a  bounded canonical factorization.

In line with Theorem~\ref{feb31} (though requiring a rather involved independent proof), Toeplitz operators $T_G$ with $G\in APW^{n\times n}$ are invertible if and only if $G$ admits a canonical $AP$ (equivalently, $APW$) factorization \cite[Section 9.4]{BKS1}. However, the necessary and sufficient conditions for $AP$ factorization, canonical or not, to exist are presently not known. The question is open even for already mentioned triangular  $2\times 2$ matrix functions \eqref{triang}. Quite a few partial results were obtained in this direction, showing that the problem is indeed intriguing and complicated. An interested reader may consult \cite{BKS1} for a coherent description of the state of affairs as of about ten years ago, and \cite{BrudRS11,CaKaS09,CaKaS10} for some more current results.

Because of these reasons, statements relating the Fredholm properties of $T_G$ and $T_{\det G}$, as well as factorization formulas for $G$, are of special interest in the $AP$ setting.

The $AP$ version of Theorem~\ref{aug112} is as follows.
\begin{thm}\label{th:ap1}Let $G\in AP^{n\times n}$ be invertible, and suppose that it contains
a submatrix  $\Psi\in (AP^+)^{(n-1)\times n}$ which is right invertible
over $AP^+$. Then the operator $T_G$ is invertible (resp. right
invertible, or left invertible) on $(H_p^+)^n$ for any (equivalently, all)
$p\in (1,\infty)$if and only if $\det G$ has zero (resp. non-positive, or
non-negative) mean motion $\kappa$. If in addition $G\in
APW^{n\times n}$ and $\kappa\geq 0$, then $G$ is $APW$ factorable,
and its partial $AP$ indices are $0,\ldots 0$ ($n-1$ times) and
$\kappa$.
\end{thm}
The proof runs along the same lines as that of Theorem~\ref{aug112},
taking into consideration that $\det G$ is an invertible $AP$ function
and thus the operator $T_{\det G}$ is automatically one sided invertible.
To construct the $APW$ factorization, one can still use formulas
\eqref{II.5}--\eqref{II.7} substituting
$r^k$ by $e_{\kappa(\det G)}$ and $\widetilde{P}^\pm$
by  the projections of $APW$ onto $APW^\pm$.

The analogue of Theorem~\ref{aug117} also holds.
\begin{thm}\label{th:ap2}Let $G\in APW^{n\times n}$ be invertible, with $\kappa(\det G)\geq 0$.
 Moreover, let there exist
$\Phi^\pm\in (APW^\pm)^{n-1\times n}$ left invertible over
$APW^\pm$ and such that $G\Phi^+=\Phi^-$. Then $G$ is $APW$
factorable, with the partial $AP$ indices equal $0,\ldots,0$ ($n-1$
times) and $\kappa(\det G)$.
\end{thm}

To state the analogue of Theorem~\ref{th:6.1}, let us introduce the notion of the $AP$ corona tuple as $$ APCT^\pm_n:=\left\{[h_1^\pm, h_2^\pm,\ldots, h_n^\pm]\colon h_j^\pm \in AP^\pm \quad \mbox{and}
\quad \inf_{z\in \C^\pm} \left(\sum_{j=1}^n |h_j^\pm (z)|\right)>0\right\}. $$

The almost periodic version of the corona theorem, in principle contained already  in \cite{ArSi56} and stated explicitly in \cite{Xia85}, reads:

Let $h_1^\pm, h_2^\pm,\ldots, h_n^\pm\in AP^\pm$. Then  $[h_1^\pm, h_2^\pm,\ldots, h_n^\pm]\in
APCT^\pm_n$ if and only if
$\left[\begin{array}{c} h_1^\pm \\ \vdots \\ h_n^\pm \end{array}\right]$ is left invertible over
$AP^\pm$.

Consequently, the $AP$ analogue of Theorem~\ref{aug93} holds.
\begin{thm}\label{APaug93}Let $\Phi\in (AP^\pm)^{n\times (n-1)}$. Then
$\Phi$ is left invertible over $AP^\pm$ if and only if
$\left[\Delta_{1,.}(\Phi),\ldots, \Delta_{n,.}(\Phi)\right]\in APCT^\pm_n. $ \end{thm}

Let now $G\in L_\infty^{n\times n}$ with all rows but one having elements in  $AP^+$ (the case of all columns but one having elements in $AP^-$ can be treated analogously). Assume for simplicity that
\[ G=\left[\begin{matrix} \Psi \\ g_n\end{matrix}\right] \text{ with } \Psi\in (AP^+)^{(n-1)\times n}, \ g_n\in L_\infty^{1\times n}. \]
Invoking Theorem~\ref{APaug93}, we immediately obtain
\begin{thm}\label{th:ap6.1} {\em (i)} If $G$ is unitary with constant determinant and $g_n^T\in APCT_n^-$, then $T_G$ is invertible for all $p\in (1,\infty)$.

{\em (ii)} If $G$ is (complex) orthogonal with constant determinant and $g_n^T\in APCT_n^+$, then $T_G$ is invertible for all $p\in (1,\infty)$.

\end{thm}

{\bf Acknowledgment.} A preliminary version of this paper was presented at the Workshop on Operator Theory and Operator Algebras (Instituto Superior T\'ecnico, Lisbon) in the Fall of 2012. The authors are thankful to Ronald Douglas and Bernd Silbermann for the helpful follow up discussion of the results.

\providecommand{\bysame}{\leavevmode\hbox to3em{\hrulefill}\thinspace}
\providecommand{\MR}{\relax\ifhmode\unskip\space\fi MR }
\providecommand{\MRhref}[2]{%
  \href{http://www.ams.org/mathscinet-getitem?mr=#1}{#2}
}
\providecommand{\href}[2]{#2}

\end{document}